\documentclass[11pt]{amsart}
\usepackage[colorlinks=true,pagebackref,hyperindex,citecolor=green,linkcolor=blue,urlcolor=blue]{hyperref}
\usepackage[top=1in, bottom=1in, left=1.4in, right=1.4in]{geometry}
\usepackage{amsmath}
\usepackage{amsfonts}
\usepackage{amssymb}
\usepackage{amsthm}
\usepackage{stmaryrd}
\usepackage[all]{xy}  
\usepackage{mathrsfs}
\usepackage{enumitem}
\usepackage[normalem]{ulem}
\usepackage{color}
\usepackage{setspace}

\makeatletter
\def\namedlabel#1#2{\begingroup
#2%
\def\@currentlabel{#2}%
\phantomsection\label{#1}\endgroup
}
\makeatother

\theoremstyle{theorem}
\newtheorem{theorem}{Theorem}[section]
\newtheorem*{theorem*}{Theorem}

\newtheorem{lemma}[theorem]{Lemma}

\newtheorem{proposition}[theorem]{Proposition}

\newtheorem{theoremx}{Theorem}

\theoremstyle{definition}
\newtheorem{definition}[theorem]{Definition}

\newtheorem{example}[theorem]{Example}

\newtheorem{remark}[theorem]{Remark}
\numberwithin{equation}{subsection}

\renewcommand{\(}{\left(}
\renewcommand{\)}{\right)}

\newcommand{\ZZ}{\mathbb{Z}}

\newcommand{\LL}{L}

\newcommand{\V}{\mathbb{V}}

\newcommand{\m}{\mathfrak{m}}
\newcommand{\n}{\mathfrak{n}}
\newcommand{\p}{\mathfrak{p}}
\newcommand{\fa}{\mathfrak{a}}

\newcommand{\q}{\mathfrak{q}}
\renewcommand{\r}{\mathfrak{r}}

\newcommand{\ideala}{\mathfrak{a}}
\newcommand{\idealb}{\mathfrak{b}}
\renewcommand{\k}{\Bbbk}
\newcommand{\kk}{K}

\newcommand{\Spec}{\operatorname{Spec}}
\newcommand{\Hom}{\operatorname{Hom}}
\newcommand{\Ext}{\operatorname{Ext}}
\newcommand{\cd}{\operatorname{cd}}
\newcommand{\InjDim}{\operatorname{inj.dim}}
\newcommand{\Supp}{\operatorname{Supp}}
\newcommand{\Depth}{\operatorname{depth}}
\newcommand{\Ker}{\operatorname{Ker}}
\newcommand{\IM}{\operatorname{Im}}	
\newcommand{\Ht}{\operatorname{height}}	
	
\newcommand{\Char}{\operatorname{char}}	
             
\newcommand{\FDer}[1]{\stackrel{#1}{\to}}

\newcommand{\onto}{\twoheadrightarrow}

\newcommand{\newLyu}{\widetilde{\lambda}}

\newcommand{\ctrOne}{t}
\newcommand{\ctrTwo}{h}

\newcommand{\down}[1]{\left\lfloor #1 \right\rfloor}

\definecolor{blue-violet}{rgb}{0.54, 0.17, 0.89}
\definecolor{Blue}{rgb}{0.01, 0.28, 1.0}
\definecolor{gGreen}{rgb}{0.2, 0.8, 0.2}
\definecolor{Green}{rgb}{0.04, 0.85, 0.32}

\begin{document}

\title[Cohomological dimension, Lyubeznik numbers, and connectivity]{Cohomological dimension, Lyubeznik numbers, and connectedness in mixed characteristic}
\author[D.\ J.\ Hern\'andez]{Daniel J.\ Hern\'andez}
\author[L.\ N\'u\~nez-Betancourt]{Luis N\'u\~nez-Betancourt}
\author[F.\ P\'erez]{Felipe P\'erez}
\author[E.\ E.\ Witt]{Emily E.\ Witt}
\maketitle


\begin{abstract}
We establish a ``second vanishing theorem'' for local cohomology modules over regular rings of unramified mixed characteristic, which relates the connectedness of the spectrum of a ring with the vanishing of local cohomology.  
Applying this, and new results on the mixed characteristic Lyubeznik numbers, 
we further study connectedness properties of the spectra of a certain class of rings.
\end{abstract}


\section{Introduction}


The goal of this paper is to relate connectedness properties of the spectrum of a ring to local cohomology modules.  
Our focus is on rings that do not contain a field, where some important known results in this direction do no apply.
An analog of the Lyubeznik numbers called the \emph{mixed characteristic Lyubeznik numbers} serve as a major tool.  


\subsection{History and motivation}
One compelling reason for studying local cohomology is its relationship with the geometry and topology of algebraic varieties.  
Recall that the \emph{cohomological dimension} of an ideal $I$ of a Noetherian ring $S$, denoted $\cd(I, S)$, is the maximum index $\ctrOne \geq0$ for which the local cohomology module $H^\ctrOne_I(S)$ is nonzero.
One of the first 
 illustrations of the noted relationship, due to Grothendieck, is that the cohomological dimension of the maximal ideal of a local ring coincides with the ring's dimension \cite{HartshorneLC,SGA2}.

Hochster and Huneke extended Faltings' connectedness theorem to link local cohomology and the topology of a ring's spectrum \cite{FalNagoya,FalPRIMS,HHgraph}.
 Given an $n$-dimensional complete local domain $S$, they showed that if the cohomological dimension of an ideal $I$ of $S$ is at most $n - 2$,  then the punctured spectrum of $S/I$ is connected.  
Crucial to their proof is the Hartshorne-Lichtenbaum vanishing theorem (often abbreviated ``HLVT'')  \cite{HartshorneCD}, which in this case says that $H^n_I(S)$ vanishes if and only if $\dim(S/I)>0$. 
Since for any ideal $I$ of $S$, $H^i_I(S) = 0$ for all $i > n$, one may regard the HLVT as a ``first vanishing theorem'' for local cohomology in the sense that it characterizes the vanishing of the local cohomology modules of $S$ at the highest index for which such a module may be nonzero.

Under the additional hypothesis that the ring $S$ contains a field, the ``second vanishing theorem''  (or ``SVT'') of local cohomology provides a converse to the aforementioned extension of Faltings' connectedness theorem:
Let $S$ be a complete regular local ring of dimension $n$ with a separably closed residue field, which it contains. Let $I\subseteq S$ be an ideal such that $\dim(S/I)\geq 2$. Then $H^{n-1}_I(S)=0$  if and only if the punctured spectrum of $S/I$ is connected 
 \cite{HartshorneCD,Ogus,P-S,H-L}.

The SVT has consequences beyond detecting whether the punctured spectrum of a ring is connected. 
For instance, it allows one to relate the Lyubeznik numbers, certain Bass numbers of local cohomology modules, to a graph that measures further connectedness properties of a ring's spectrum \cite{Walther2,LyuInvariants,W}.  
It is also a key ingredient in showing that for an ideal $I$ of a polynomial ring $S$ in $n$ variables over a field of characteristic zero, if 
$\Depth(S/I) \geq 3$, then  $\cd(I, S)\leq n-2$ \cite{VarbaroCD}.


\subsection{A second vanishing theorem in mixed characteristic}
It is natural to ask whether there is an analogous SVT characterizing the connectedness of the punctured spectra of rings containing no field.
The Cohen structure theorems allow one to write any complete local ring of mixed characteristic as the quotient of an unramified regular local ring of mixed characteristic \cite{Cohen}.
In light of this, our first goal (and the focus of Section \ref{second-vanishing: S}) is the following.


\begin{theoremx}[SVT over unramified rings of mixed characteristic; cf.\ Theorem \ref{SVTUnramifiedGeneral: T}] \label{SVTUnramified: T}
Let $S$ be an $n$-dimensional complete unramified regular local ring of mixed characteristic, which has a separably closed residue field. 
Let $I$ be an ideal of $S$ for which $R=S/I$ is equidimensional, and $\dim(R) \geq 3$. 
Then
$H^{n-1}_I(S)=0$ if and only if the punctured spectrum of $R$ is connected.
\end{theoremx}


\noindent Please see Theorem \ref{SVTUnramifiedGeneral: T} for a more general statement. 

We note that when recently announcing this result at a conference,\footnote{\href{https://sites.google.com/site/dmodulesincommutativealgebra/home}{Conference on $D$-modules in Commutative Algebra}, Centro de Investigaci\'on en Matem\'aticas (CIMAT) in Guanajuato, Mexico from August 10 - 14, 2015.} 
we learned from W.\ Zhang that he has independently proven Theorem \ref{SVTUnramified: T} using different techniques. 

Although our principal motivation for seeking such a theorem is to better understand the topology of a ring's spectrum,  cohomological dimension is itself an influential research topic.
It has connections with, for instance, the number of equations defining a variety \cite{BroSharp,TwentyFourHours}, depth \cite{P-S,LyuVan,VarbaroCD,DaoTakagi}, and the vanishing of singular and algebraic de Rham cohomology \cite{GarciaSabbah, LyuLC, SwitalaNonsingular}. 
The new SVT may be a useful tool for extending results previously known only in equal characteristic to the mixed characteristic setting. 
For example, we apply it in proving Theorem \ref{cdSVTconsequence: T}, a mixed characteristic version of a theorem of Huneke and Lyubeznik on cohomological dimension \cite{H-L}. 


\subsection{Connectedness of spectra and Lyubeznik numbers}
The Lyubeznik numbers are numerical invariants associated  to a local ring $R$ containing a field, defined via local cohomology \cite{LyuDMod}.   
If $R$ has a separably closed residue field, its highest Lyubeznik number equals the number of connected components of its Hochster-Huneke graph \cite{LyuInvariants, W}.  
Assuming that $R$ is equidimensional, this graph is connected if and only if the connectedness dimension of $R$ is at least $\dim(R) -1$,  i.e., 
$\Spec(R)\setminus Z$ is connected for every closed set $Z$ of $\Spec(R)$ for which $\dim(Z) \leq \dim(R)-2$ \cite{HHgraph}.  
If $R$ is a complete intersection on each point of its punctured spectrum, the highest Lyubeznik number, in fact, equals the number of connected components of the punctured spectrum of $R$ \cite{GarciaSabbah, B-B,BlickleCCI}.

Our second objective (and the focus of Section \ref{SecIsoSing: S}) is to show that the \emph{mixed characteristic Lyubeznik numbers}, recently-defined numerical invariants of local rings whose residue fields have prime characteristic,
are related to connectedness properties of spectra \cite{NuWiMixChar}.  We prove the following. 


\begin{theoremx}[{cf.\ Theorem \ref{connectedComponentsLyu: T}}]\label{connectedComponentsLyuIntro: T}
Let $R$ be a complete local ring that is a complete intersection at every point of its punctured spectrum, with a separably closed residue field of prime characteristic, and for which $\dim(R) \geq 4$. 
Then the highest mixed characteristic Lyubeznik number of $R$ equals the number of connected components of its punctured spectrum.
\end{theoremx}


 Our result is actually more general, applying to \emph{cohomologically complete intersection ideals} on the punctured spectrum of a ring, a notion introduced by Hellus and Schenzel \cite{CCI}.
Toward the same goal, we also show that when the highest mixed characteristic Lyubeznik is as small as possible (\emph{i.e.}, equals one), then
the spectrum of the ring is ``highly connected.''  We refer the reader to Theorem \ref{connectednessDimensionLyu: T} for our precise statement.

\section{Background} \label{Sec: Background}


\subsection{Connectedness properties of spectra}  \label{Connectedness: SS}


 Let $(S, \n)$ be a local ring, fix an ideals $I \subseteq J$ of $S$, and let $X = \Spec(S/I)$. 
Recall that open subset of a topological space is connected if it cannot be represented as the union of two disjoint, nonempty open sets.  
Translating this condition algebraically, for an ideal $J$ of $S$,  $X \setminus \mathbb{V}(J)$ is connected in the Zariski topology if and only if the following property holds:
given ideals $\ideala$ and $\idealb$ of $S$ for which 
\[
\sqrt{\ideala \cap \idealb} = \sqrt{I} \  \text{and}  \ \sqrt{\ideala + \idealb} = \sqrt{J}  \ \text{in}  \ S,
\]
one of $\sqrt{\ideala}$ and $\sqrt{\idealb}$ must be $\sqrt{J}$ (and equivalently, one of these must be $\sqrt{I}$).  
Indeed, $\V(\ideala)$ and $\V(\idealb)$ are disjoint, nonempty open and closed sets of $X \setminus \mathbb{V}(J)$, whose union is this entire space.

In particular, the punctured spectrum of $\Spec^\circ(S/I)$ is connected if and only if whenever $\ideala$ and $\idealb$ are ideals of $S$ for which $\sqrt{\ideala \cap \idealb} = \sqrt{I}$ and $\sqrt{\ideala + \idealb}  = \n$, one of  $\ideala$ or $\idealb$ must be $\n$-primary.

The \emph{connectedness dimension} of a Noetherian ring $R$, denoted $c(R)$, is defined as
\[
c(R) = \min\{ \dim(Z) \mid Z \subseteq \Spec(R) \ \text{closed and} \ \Spec(R) \setminus Z \ \text{ disconnected} \}
\]
\cite[Definition 19.1.9]{BroSharp}.  Note that since $\Spec(R) \setminus \Spec(R) = \emptyset$ is (trivially) disconnected, $c(R) \leq \dim(R)$.
We refer the reader to \cite[Chapter 19]{BroSharp} for more on the study of connectedness dimension via local cohomology, and in particular, on cohomological dimension.

\subsection{Long exact sequences in local cohomology} \label{Loc Coh} 
Let $R$ be a Noetherian ring.  Given an ideal $I$ of $R$, fix generators $f_1, \ldots, f_\ctrOne$ for $I$.  
The \emph{$j^\text{th}$ local cohomology module of an $R$-module $M$ with support in $I$}, denoted $H^j_I(M)$, can be computed as the $j^\text{th}$ cohomology module of a \v{C}ech-like complex
\[
0\to M\to \bigoplus_{i=1}^\ctrOne M_{f_i}\to \bigoplus_{1 \leq i<\ell \leq \ctrOne} M_{f_i f_{\ell}}\to \ldots \to M_{f_1 \cdots f_\ctrOne}\to 0,
\]
where the homomorphism on each summand, up to a sign, is the appropriate further localization map.

Local cohomology modules serve both as the major objects of study, and as important tools, in this paper.  
In order to understand them, we apply three types of long exact sequences associated to these modules.  
First, given an ideal $I$ of $R$ and a short exact sequence of $R$-modules $0 \to M \to N \to P \to 0$,
there is a functorial long exact sequence of the form
\[ \cdots \to H^j_I(M) \to  H^j_I(N) \to  H^j_I(P) \to H^{j+1}_I(M) \to \cdots. \]
Given $f \in R$ and an $R$-module $M$, there is also a long exact sequence, functorial in $M$, 
\begin{equation*} \label{LCLESelement}
\cdots \to H^j_{I+ f R}(M) \to H^j_I(M) \to H^j_I(M_f) \to H^{j+1}_{I+ f R}(M) \to\cdots,
\end{equation*}
where the map from $H^j_I(M)$ to $H^j_I(M_f) \cong H^j_I(M)_f$ is the natural localization map.
Finally, the Mayer-Vietoris sequence in local cohomology is especially useful in the study of connectedness properties of $\Spec(R)$:  given ideals $I$ and $J$ of $R$, and $R$-module $M$, there is a long exact sequence of the form
\begin{equation*} \label{LC-LES-MV}
\cdots \to H^j_{I+ J}(M) \to H^j_I(M)\oplus H^j_J(M)\to H^j_{I\cap J} (M) \to\cdots.
\end{equation*}


\subsection{Lyubeznik numbers and mixed characteristic Lyubeznik numbers.} \label{Subsec: LyuNumbers} 
Given a local ring $(R,\m, K)$ containing a field, the Cohen structure theorems allow us to write $\widehat{R} \cong S / I$, where $I$ is an ideal of a complete regular local ring $(S, \n, K)$ containing a field.  
Given integers $i, j \geq 0$, if $S$ is $n$-dimensional, then the  \emph{Lyubeznik number of $R$ with respect to $i$ and $j$} is
\[ \lambda_{i,j}(R)=\dim_K \Ext^i_S(K,H_I^{n-j}(S)), \]
which depends only on $R$, $i$, and $j$ \cite{LyuDMod}. 
Note that $\lambda_{i,j}(R)$ is the $i^\text{th}$ Bass number of $H_I^{n-j}(S)$ with respect to $\n$, which is finite \cite{Huneke,LyuDMod}. 
If $d = \dim(R)$, then $\lambda_{d, d}(R)\neq 0$ and $\lambda_{i, j}(R) = 0$ if either $i > d$ or $j> d$ \cite{LyuDMod}.  
For this reason, $\lambda_{d, d}(R)$ is called the \emph{highest Lyubeznik number} of $R$.

Now take any local ring $(R,\m,K)$ for which $K$ has positive characteristic; i.e., $R$ does not necessarily contain a field.  In this setting, we can write $\widehat{R} \cong T/I$, where $I$ is an ideal of an unramified regular local ring $(T, \n, K)$ of mixed characteristic \cite{Cohen}.  
For integers $i, j \geq 0$, if $n = \dim(T)$, the \emph{mixed characteristic Lyubeznik number with respect to $i$ and $j$} is the $i^\text{th}$ Bass number of $H_I^{n-j}(T)$ with respect to $\n$,
\[
\newLyu_{i,j}(R)=\dim_K \Ext_T^i(K,H_I^{n-j}(T)) 
\]
\cite{NuWiMixChar}, which is also finite \cite{LyuUMC, NunezPR}.
Again, this number is independent of the presentation of $\widehat{R}$.  
If $d = \dim(R)$, the \emph{highest mixed characteristic Lyubeznik number} is $\newLyu_{d,d}(R)$; indeed, $\newLyu_{i,j}(R) = 0$ if either $i$ or $j$ exceed $d$ \cite{NuWiMixChar}.

It is important to note that for local rings of characteristic $p$, both families of numerical invariants are defined; there are many circumstances in which they agree, but they do not always coincide \cite{NuWiMixChar, InjDim}.

\section{A second vanishing theorem over unramified  regular rings of \\ mixed characteristic} \label{second-vanishing: S}


The primary goal of this section is to prove a second vanishing theorem (SVT) for local cohomology modules over unramified regular local rings of mixed characteristic. 
Like the SVT over regular local rings containing a field, it characterizes the connectedness of punctured spectra through the vanishing of local cohomology.


\subsection{Two graphs}
In this subsection, we recall some facts about two graphs associated to a local ring.  
Both are simple graphs, meaning that they are undirected, contain at most one edge between any two vertices, and have no loops.

A particular simplicial complex has been used to relate connectedness, cohomological dimension, and depth \cite{LyuLC,ExtHar}.  
If a local ring $(R,\m)$ has minimal primes $\p_1, \ldots, \p_t$, then this simplicial complex has vertices labeled $1, \ldots, t$;  
$\{ i_1, \ldots, i_j \} \subseteq [t]$ is a face of  $\Delta(R)$ precisely if $\p_{i_1} + \cdots + \p_{i_j}$ is not $\m$-primary.
The first graph we consider is the one-skeleton of $\Delta(R)$.


\begin{definition}[The graph $\Theta_R$]
Given a local ring $(R, \m)$ with minimal primes $\p_1, \ldots, \p_\ctrOne$, the graph $\Theta_R$ has vertices labeled $1, \ldots, \ctrOne$, and 
there is an edge between two (distinct) vertices $i$ and $j$ precisely if $\p_i + \p_j$ is not $\m$-primary. 
\end{definition}


Our proof of the new SVT exploits the fact that this graph detects connectedness of the punctured spectrum, which 
Huneke and Lyubeznik originally pointed out in the proof of \cite[Theorem 2.9]{H-L}.

\begin{lemma}[Huneke\,--\,Lyubeznik]  \label{ThetaConnected: L}
Given a complete local ring $R$, the graph $\Theta_R$ is connected if and only if the punctured spectrum of $R$ is connected.
\end{lemma}
Note that the condition that $\Theta_R$ is connected is sometimes written as the equivalent condition that the reduced homology group $\widetilde{H}_0(\Delta(R),\ZZ)$ vanishes (see \cite{LyuLC,ExtHar}).


Along with the graph $\Theta_R$, we also use what is now often called the \emph{Hochster-Huneke graph} \cite[Definition 3.4]{HHgraph}.


\begin{definition}[The Hochster-Huneke graph $\Gamma_R$]
Given a local ring $R$, list all the minimal primes $\p$ of $R$ for which $\dim(R)=\dim(R/\p)$ as $\p_1, \ldots, \p_\ctrTwo$.
The \emph{Hochster}-\emph{Huneke graph of $R$}, denoted $\Gamma_R$, has vertices $1, \dots, \ctrTwo$;
there is an edge between (distinct) vertices $i$ and $j$ precisely if $\p_i+\p_j$ has height one.
\end{definition}

\noindent The Hochster-Huneke graph of a complete equidimensional local ring $R$ is connected if and only if $\Spec(R) \setminus \V(J)$ is a connected topological space for all ideals $J$ of $R$ of height at least two \cite[Theorem 3.6]{HHgraph}.  
In particular, the graph $\Gamma_R$ is disconnected if the punctured spectrum of $R$ is a disconnected topological space.


 The number of connected components of the Hochster-Huneke graph is closely related to the highest Lyubeznik number, as we now describe.  
We take advantage of this relationship in Section \ref{SecIsoSing: S}.

\begin{remark}\label{Rem HHgraph}
If $R$ is a complete $d$-dimensional local ring containing a field, then the highest Lyubeznik number $\lambda_{d,d}(R)$ counts the number of components of the Hochster-Huneke graph of $\widehat{R}^{sh}$, the completion of the strict Henselization of $R$ \cite[Theorem 1.3]{LyuInvariants},\ \cite[Main Theorem]{W}. 
Suppose that $\Gamma_R$ has $t$ components, and let $J_i$ be the intersection of the minimal primes of $I$ that are vertices of the $i^\text{th}$ component of $\Gamma_R$.
If $S$ is an $n$-dimensional regular local ring containing a field and $R=S/I$, then $H^{n-d}_I(S)\cong H^{n-d}_{J_1}(S) \oplus \cdots \oplus  H^{n-d}_{J_t}(S)$  \cite[Proposition 2.1]{LyuInvariants}.
This means that  $\lambda_{d,d}(R)= \lambda_{d,d}(S/J_1) + \cdots + \lambda_{d,d}(S/J_\ctrOne)$, and since each $\lambda_{d,d}(S/J_i) \geq 1$ \cite[(4.4iii)]{LyuDMod}, we have that $\lambda_{d,d}(R) \geq t$; i.e.,  the number of connected components of $\Gamma_{\widehat{R}^{sh}}$ is 
at least the number of connected components of $\Gamma_R$.

Lyubeznik pointed out that it follows from \cite[Theorem 4.2]{H-L} that if $\k$ is a coefficient field of $R$, then there exists a finite separable extension field $\LL$ of $\k$ for which
$\Gamma_{R\otimes_\k \LL}$ and $\Gamma_{\widehat{R}^{sh}}$ are isomorphic graphs \cite[p.\,630]{LyuInvariants}.
Then  since the strict Henselization of $R \otimes_\k \LL$ is ${R}^{sh}$, $\lambda_{d,d}(R \otimes_\k \LL)=\lambda_{d,d}(R)$ is the number of connected components of  $\Gamma_{R\otimes_\k \LL}$.
\end{remark}



A key ingredient in the proof of the new SVT is the behavior of the Hochster-Huneke graph modulo a parameter.
The following result was originally proved by W.\ Zhang \cite[Proposition 2.2]{W}.
It appears as a step toward the equicharacteristic main result of the paper it appears in \cite[Main Theorem]{W}, and several blanket hypotheses are fixed throughout its section of the manuscript.  
Due to this, and because the main result's proof is not presented in a linear manner, 
we include a proof here that follows Zhang's argument to convince the reader that the assumptions below suffice.
In  particular, no hypothesis on characteristic is necessary.


\begin{proposition}[cf.\,{\cite[Proposition 2.2]{W}}]\label{HHGraphParameter: P}
If $(R, \m)$ is a complete local domain and $x \in \m$ is nonzero, then the Hochster-Huneke graph of $R/xR$ is connected. 
\end{proposition}


\begin{proof}
Let $A = R /\sqrt{xR}$.  Since the ideals $xR$ and $\sqrt{x R}$ of $R$ have the same minimal primes, the Hochster-Huneke graphs of $R/xR$ and $A$ are isomorphic, so 
it suffices to show that $\Gamma_A$ is connected.

Let $S$ be the integral closure of $R$ in its field of fractions.
Then $S$ is a complete local domain, and is a finitely-generated $R$-module \cite[Theorem 4.3.4]{SwHu}.
Let $B=S/\sqrt{xS}$.  Since $S$ satisfies Serre's criterion $S_2$, the $S_2$-ification of $S/\sqrt{xS}$ is local \cite[Proposition 3.9]{HHgraph}, so that 
 $\Gamma_{B}$ is connected \cite[Theorem 3.6]{HHgraph}.

We have that $\sqrt{xS}\cap R=\sqrt{xR}$ by the going up theorem, so $A$ injects into $B$, and $B$ is a finitely-generated $A$-module, so that $A$ and $B$ have the same dimension.  
Fix minimal primes $\p$ and $\p'$ of $A$.  Again by the going up theorem, there exist minimal primes $\q$ and $\q'$ of $B$ for which $\q \cap A=\p$ and $\q' \cap A =\p'$. 

Since $\Gamma_B$ is connected, there exists a finite sequence of minimal primes  
\[\q=\q_1, \q_2,\ldots,\q_\ctrOne =\q'\]
 of $B$ for which $\Ht_B( \q_i+\q_{i+1}) = 1$ for $1 \leq i < \ctrOne$. 
For each $1 \leq i \leq n$, let $\p_i=\q_i \cap A$.
We aim to show that for all $1 \leq i < \ctrOne$, $\Ht_A(\p_i+\p_{i+1}) \leq 1$ (i.e., either $\p_i = \p_{i+1}$ or $\Ht_A(\p_i+\p_{i+1}) = 1$), so that $\Gamma_A$ is connected, completing the proof.

For some $1 \leq i < \ctrOne$, let $\r$ be a prime ideal of $B$ of height one containing $\q_i+\q_{i+1}$. 
Let $d = \dim(R)$.
Then $\dim(B) = d -1$, and since $B$ is equidimensional and catenary, there exists a chain of $d-2$ primes ideal of $A$ containing $\r$,
\[
\r \subsetneq \r_2 \subsetneq \ldots \subsetneq \r_{d-1} \subseteq  B.
\]
Then by the lying over theorem, we also have the strict chain
\[
\r\cap A \subsetneq \r_2 \cap A \subsetneq \ldots \subsetneq \r_{d-1}\cap A \subseteq A.
\] 
Since $\dim(A)=d-1$, the $\Ht_A(\r \cap A) \leq 1$.
Therefore, the $\Ht_A(\p_i+\p_{i+1}) \leq 1$, and the result follows.
\end{proof}


\subsection{Second vanishing theorem over unramified  regular local rings of mixed characteristic}
In proving the new SVT, we have the following strategy:  
Like Huneke and Lyubeznik's proof of the SVT in equal characteristic \cite[Theorem 2.9]{H-L}, we first study local cohomology with support in a prime ideal. 
For arbitrary ideals, we then use an inductive argument on the number of minimal primes of the ideal.
The heart of our proof lies in the technique we apply to the case of prime ideals, where are are able to avoid machinery like that in \emph{loc.\,cit.} by reducing to the equal-characteristic case addressed there.  
Toward this, consider the following lemma.


\begin{lemma} \label{cohomologicalDimensionReduction: L}  
Let $S$ be a complete unramified  regular local ring of mixed characteristic $p>0$, and let $A=S/pS$. 
If $I$ is an ideal of $S$ for which $p \in I$, then \[\cd(S, I)\leq \cd(A, IA)+1.\]
\end{lemma}


\begin{proof}
Fix $i > \cd(A, IA) + 1$; we aim to show that $H^i_I(S) = 0$.
Indeed, $H^{i-1}_{IA}(A) = 0$ for such a choice of $i$, so that the long exact sequence in local cohomology 
\[
\cdots \to H^{i-1}_{IA}(A) \to H^i_I(S)\FDer{p}H^i_I(S)\to \cdots
\]
associated to the short exact sequence $0\to S\FDer{p} S\to A\to 0$
then implies that multiplication by $p$ is injective on $H^i_I(S)$.

However, this is impossible unless $H^i_I(S)$ itself vanishes:  If not, fix the smallest positive integer $N$ for which a nonzero element $u \in H^i_I(S)$ is killed by $p^N$, which exists since $p \in I$.  
Then $p^{N-1} u$ is a nonzero element in the kernel of the map given by multiplication by $p$ on $H^i_I(S)$, a contradiction. 
\end{proof}


We can now address the case of local cohomology with support in a prime ideal.  


\begin{lemma} \label{SVTPrime: L}
Let $S$ be an $n$-dimensional complete unramified  regular local ring of mixed characteristic, with a separably closed residue field.
If $\q$ is a prime ideal of $S$ for which $\dim (S/\q)\geq 3$, then $H^{n-1}_\q(S)=0$.
\end{lemma}

\begin{proof}
Suppose that $S$ has mixed characteristic $p>0$, and assume first that $p\in \q$.  
If $A = S/pS$, then $S/\q \cong A/\q A$ has equal characteristic $p$.
Since $\dim\( A/\q A \) = \dim\( S/\q \) \geq 3$, $\q A$ is not primary to the maximal ideal of $A$.
Hence, since $\dim(A) = n-1$,  $H^{n-1}_{\q A}(A)$ vanishes by the HLVT \cite[Theorem 3.1]{HartshorneCD}.
Moreover, $H^{n-2}_{\q A}(A) = 0$ by the SVT in equal characteristic \cite[Theorem 2.9]{H-L}.
Thus, $\cd(A, \q A) \leq n-3$, so that by Lemma \ref{cohomologicalDimensionReduction: L}, $\cd(S, \q) \leq n-2$, finishing the proof in this case.

Now suppose, alternatively, that $p\not\in \q$.  
Let $J=\q+pS$, so that $S/J$ has equal characteristic $p$ and dimension at least two.
For distinct minimal primes $\p_1$ and $\p_2$ of $S/J$, if $\Ht_{S/J}(\p_1 + \p_2) = 1$, then 
$\p_1+\p_2$ cannot be primary to the maximal ideal of $S/J$ since $\dim(S/J) \geq 2$.  
Thus, $\Gamma_{S/J}$ is a subgraph of $\Theta_{S/J}$.
In fact, it is a spanning subgraph:  
since $S/\q$ is catenary, $S/J$ is equidimensional by Krull's principal ideal theorem; therefore, the vertices of $\Gamma_{S/J}$ are indexed by all minimal primes of $S/J$, like that of $\Theta_{S/J}$. 

Since $\q$ is prime and $S$ is complete, $S/\q$ is a complete local domain.
Therefore, taking $x = p$ in Proposition \ref{HHGraphParameter: P}, we know that the Hochster-Huneke graph $\Gamma_{S/J}$ is connected.
Since $\Gamma_{S/J}$ is a spanning subgraph of $\Theta_{S/J}$ (i.e., the two graphs have the same vertices), this subgraph must also be connected.  
Therefore, the punctured spectrum of $S/J$ is connected by Lemma \ref{ThetaConnected: L}.
Then by the SVT in equal characteristic, $\cd(A, JA) \leq n-3$ \cite[Theorem 2.9]{H-L}, so that $\cd(S, J)\leq n-2$ by Lemma \ref{cohomologicalDimensionReduction: L}, and, in particular, $H^{n-1}_J(S)=0$.
With this in mind, in order to show that $H^{n-1}_\q(S)$ vanishes, by
 the long exact sequence
\[
\ldots \to H^{n-1}_J(S)\to H^{n-1}_\q(S)\to H^{n-1}_\q(S_p)\to 0,
\]
it suffices to show that $H^{n-1}_\q(S_p)$ does.

Toward this, note that the dimension of $S_p$ is $n-1$ since $p$ is in the maximal ideal of $S$, but not in every prime ideal of height $n-1$ of $S$.  
By similar reasoning, $\dim\(S_p/\q S_p\) \geq 2$.
Moreover, for every prime ideal $\p$ of $S$ not containing $p$, $S_\p$ is also a complete regular local ring, and $\( S_p \)_\p \cong S_\p$, so that $\( H^{n-1}_\q(S_p) \)_\p \cong H^{n-1}_{\q S_\p} (S_\p) = 0$ by the HLVT.  Thus, $H^{n-1}_\q(S_p) = 0$, completing the proof.
\end{proof}


We now are prepared to prove an SVT over unramified  regular local rings of mixed characteristic.  
Our statement is more general than Theorem \ref{SVTUnramified: T} from the introduction.


\begin{theorem}[SVT over unramified  regular rings of mixed characteristic; cf.\ Theorem \ref{SVTUnramified: T}] \label{SVTUnramifiedGeneral: T}
Let $S$ be an $n$-dimensional complete unramified  regular ring of mixed characteristic, whose residue field is separably closed. 
Let $I$ be an ideal of $S$ for which $\dim(S/\p) \geq 3$ for every minimal prime $\p$ of $I$.
Then $H^{n-1}_I(S)=0$ if and only if the punctured spectrum of $S/I$ is connected.
\end{theorem}


\begin{proof}
First assume that $H^{n-1}_I(S)=0$, and by way of contradiction, assume that the punctured spectrum of $S/I$ is disconnected.
Let $\n$ denote the maximal ideal of $S$, so that there exist ideals $J_1$ and $J_2$ of $S$ that are not $\n$-primary
for which $\sqrt{J_1\cap J_2}=\sqrt{I}$ and $\sqrt{J_1+J_2}=\n$. 
Consider the Mayer-Vietoris sequence associated to these ideals,
\[
\cdots \to H^{n-1}_I(S)\to H^n_{J_1+J_2}(S)\to H^n_{J_1}(S)\oplus H^n_{J_2}(S)\to H^n_{I}(S)\to 0.
\]
By our choice of $I$,  $H^{n-1}_I(S)=0$, and $H^n_{J_1}(S)=H^n_{J_2}(S)=H^n_{I}(S)=0$
by the HLVT  \cite[Theorem 3.1]{HartshorneCD}.
Then $H^n_\n(S)=H^n_{J_1+J_2}(S)=0$, contradicting the HLVT.

For the other direction, we proceed by induction on the number of minimal primes of $I$.
For ideals with only one minimal prime, the result follows from Lemma \ref{SVTPrime: L}.
Fix an integer $t >1$.
Assume that for all ideals $J$ of $S$ with $\ctrOne-1$ minimal primes, for which $\dim(S/\p) \geq 3$ for each minimal prime $\p$ of $J$,
if the punctured spectrum of $S/J$ is connected, then $H^{n-1}_J(S) = 0$.

Fix an ideal $I$ with $\ctrOne$ minimal primes, such that $\dim(S/\p) \geq 3$ for each minimal prime $\p$ of $I$, and for which the punctured spectrum of $R:=S/I$ is connected.
Then $\Theta_R$ is also connected by Lemma \ref{ThetaConnected: L}. 
Thus, there is an ordering $\q_1,\ldots, \q_\ctrOne$ of the minimal primes of $I$ such that the induced subgraph of $\Theta_R$ on indices $\q_1, \ldots, \q_i$ (i.e., the graph on these vertices, with all edges between them in $\Theta_R$) is connected for all $1 \leq i \leq \ctrOne$.   
This means that given $1 \leq i \leq \ctrOne$, if  $J_i=\q_1\cap \ldots \cap \q_i$, then the graph $\Theta_{S/J_i}$ is connected.
Again calling upon Lemma \ref{ThetaConnected: L}, we deduce that the punctured spectrum of each $S/J_i$ is connected.
Consider the following piece of the Mayer-Vietoris sequence in local cohomology associated to $J_{\ctrOne-1}$ and $\q_\ctrOne$:
\[
\ldots \to H^{n-1}_{J_{\ctrOne-1}}(S)\oplus H^{n-1}_{\q_\ctrOne}(S)\to H^{n-1}_{I}(S)\to H^{n}_{J_{\ctrOne-1}+\q_\ctrOne}(S)\to 0.
\]
By the inductive hypothesis, $H^{n-1}_{J_{\ctrOne-1}}(S)=0$. 
In addition, $H^{n-1}_{\q_\ctrOne}(S)=0$ by Lemma \ref{SVTPrime: L}.
Moreover, since the punctured spectrum of $R$ is connected, $\sqrt{J_{\ctrOne-1}+\q_\ctrOne} \subsetneq \n$, so that $H^n_{J_{\ctrOne-1}+\q_\ctrOne}(S)=0$ by the HLVT \cite[Theorem 3.1]{HartshorneCD}. 
Hence, $H^{n-1}_I(S)=0.$
\end{proof}


\subsection{Consequences of the new SVT}

As a by-product of the new SVT, we can establish a mixed characteristic version of a result first obtained by Huneke and Lyubeznik in equal characteristic \cite[Theorem 5.2]{H-L}. 
The proof follows the same general argument as in \emph{loc.\,cit.}, but we include the details for completeness. 

Before proving this theorem, we state a result of Faltings \cite[Korollar 2]{Fa1}. 
This theorem is stated only for rings that contain a field; however, as he mentions in the introduction of \emph{loc.\,cit.}, the result also holds in the mixed characteristic setting.
We recall that the \emph{big height} of an ideal is the maximum of the heights of its minimal primes.
We use the notation $\down{ x }$ to denote the greatest integer not exceeding a real number $x$. 


\begin{theorem}[Faltings]  \label{Faltings: T} 
Let $S$ be an $n$-dimensional complete regular ring whose residue field is separably closed.  
If $I$ is an ideal of $S$ with big height $b \geq 1$, then \[\cd(S, I) \leq n- \down{ n/b }.\]   

\end{theorem} 


\noindent The following is an adaption of a result of Lyubeznik to the mixed characteristic setting, which we apply in the subsequent theorem.


\begin{lemma}[cf.\ {\cite[Lemma]{LySomeAlgebraicSets}}] \label{Lyubeznik 85: L}
Let $S$ be an $n$-dimensional unramified complete regular ring of unramified mixed characteristic, with a separably closed residue field.
Given ideals $I_1,\ldots,I_t$ of $S$, each of big height $b \geq 1$,
 \[\cd(S , I_1 + \cdots + I_t)<n-\down{n/b}+t.\]   
\end{lemma} 

\begin{proof}
We proceed by induction on $t$.  
If $t = 1$, the statement follows by Theorem \ref{Faltings: T} \cite[Korollar 2]{Fa1}.  
Fix $t > 1$, and assume that for any $1 \leq s < t$, given ideals $J_1, \ldots,  J_s$ of $S$ with big height $b\geq 1$, we have that $\cd(S,  J_1 + \cdots + J_s) < n - \down{n/b}+s$.

Fix ideals  $I_1,\ldots,I_t$ of $S$ of big height $b$, and let 
\[ I'=I_1+\cdots+I_{t-1}, \ I=I' + I_t, \text{ and } \ I''=I_1\cap I_t+\cdots+I_{t-1}\cap I_t.\] 
Since $\sqrt{I''}=\sqrt{I'\cap I_t}$, the Mayer-Vietoris sequence associated to $I'$ and $I_t$ has the form
\begin{equation} \label{MVsum: e}
\cdots \to H^i_{I''}(S)\to H^{i+1}_I(S) \to H^{i+1}_{I'}(S)\oplus H^{i+1}_{I_t}(S)\to \cdots.
\end{equation}
By our inductive hypothesis, $H^i_{I'}(S)=H^i_{I''}(S)=0$ for all $i\geq n-\down{n/b}+(t-1)$, and $H^i_{I_t}(S)=0$ for $i \geq n-\down{n/b}+1$.
Therefore, taking $i \geq n-\down{n/b}+(t-1)$ in \eqref{MVsum: e},  we conclude that $H^i_I(S)=0$ for all $i\geq n- \down{n/b}+t$.
\end{proof}


\noindent We can now prove the aforementioned version of the theorem of Huneke and Lyubeznik.


\begin{theorem}[cf.\ {\cite[Theorem 5.2]{H-L}}] \label{cdSVTconsequence: T}
Let $S$ be an $n$-dimensional complete unramified  regular local ring of mixed characteristic, with a separably closed residue field.
For an integer $1 \leq b < n$, let $\ctrOne = \down{(n-1)/b}$.  
Fix ideals $I_1,\ldots,I_\ctrOne$ of $S$, each of big height $b$, 
and let $J = I_1 + \cdots + I_\ctrOne$.
Suppose that the punctured spectrum of $S/J$ is disconnected, 
and that for every minimal prime $\p$ of $J$, $\dim(S/\p) \geq 3$.  
If $I=I_1\cap\cdots\cap I_\ctrOne$, then \[\cd(S, I)=n-\ctrOne.\]
\end{theorem}


\begin{proof}
For $1 \leq i \leq \ctrOne$, let  $\ideala_i= I_1 \cap \cdots \cap I_i + I_{i+1} + \cdots +  I_\ctrOne$.  
As $b < n$, we have that $\down{ (n-1)/b } = \down{ n/b }$.
Then by Lemma \ref{Lyubeznik 85: L}, 
\begin{equation} \label{cdBound: e}
\cd(S, \ideala_i) < n- \down{n / b} - (\ctrOne - i + 1) = n - i + 1.
\end{equation} 
We proceed by induction on $i \geq 1$ to show that $\cd(S, \ideala_i)=n-i$, so that our desired statement is obtained by taking $i=\ctrOne$.

When $i=1$, we have that $\cd(S, \ideala_1) = n-1$ by Theorem \ref{SVTUnramifiedGeneral: T} since the punctured spectrum of $S/\ideala_1$ is disconnected by assumption.
Suppose that for some $1 < i \leq \ctrOne$, $\cd(S, \ideala_{i-1}) = n-i+1$.
Let 
\[
I'=I_i+\cdots + I_\ctrOne \ \text{ and } \ I''=I_1 \cap \cdots \cap I_{i-1} + I_{i+1}+\cdots+ I_\ctrOne. \]
The Mayer-Vietoris sequence associated to $I'$ and $I''$ has the form
\[
\cdots \to H^{n-i}_{\ideala_i}(S) \to H^{n-i+1}_{\ideala_{i-1}}(S)  \to H^{n-i+1}_{I'}(S) \oplus  H^{n-i+1}_{I'}(S) \to \cdots.    
\]
since $\ideala_{i-1}=I'+I''$ and $\sqrt{ \ideala_i }= \sqrt{ I'\cap I''}$.
By Lemma \ref{Lyubeznik 85: L}, the cohomological dimensions of both $I'$ and of $I''$ are less than $n-i+1$, so
\[H^{n-i+1}_{I'}(S) = H^{n-i+1}_{I'}(S) = 0.\]
Moreover, $H^{n-i+1}_{\ideala_{i-1}}(S) \neq 0$ by the inductive hypothesis.
Thus, $H^{n-i}_{\ideala_i}(S)$ is nonzero, and noting \eqref{cdBound: e}, we have that $\cd(S, \ideala_i) = n - i$.
\end{proof}

Our final application of the new SVT in this section allows us to conclude that a 
 certain mixed characteristic Lyubeznik number determines the number of connected components of the punctured spectrum of a ring. 
The analogous statement for traditional Lyubeznik numbers in the equal characteristic case is well known to experts; see \cite[Proposition 3.1]{Walther2} or \cite[Theorem 1]{Kawasaki2} for a proof in this case in the two-dimensional setting.


\begin{proposition}  \label{connectedComponents: P}
Let $S$ be an $n$-dimensional complete unramified  regular local ring of mixed characteristic, with separably closed residue field $K$. 
Let $I$ be an ideal of $S$ such that for every minimal prime $\p$ of $I$, $\dim(S/\p) \geq 3$.
If the punctured spectrum of $S/I$ has $t$ connected components, then  
$H^{n-1}_I(S)=E_S(K)^{\oplus t-1}$.
Thus, \[\newLyu_{0,1}(S/I)=t-1.\]
\end{proposition}


\begin{proof}
Fix ideals $J_1,\ldots,J_t$ of $S$ for which the complements of $\V(J_1), \ldots, \V(J_t)$ are precisely the connected components of the punctured spectrum of $S/I$; i.e., each complement is nonempty, they do not pairwise intersect, and their union is $\Spec^\circ(S/I)$.

We proceed to show that $H^{n-1}_I(S)=E_S(K)^{\oplus t-1}$ by induction on $t$, noting that the result follows from Theorem \ref{SVTUnramifiedGeneral: T} if $t=1$.
Fix $t > 1$, and assume that for any ideal $J$ of $S$ for which $\dim(S/\q) \geq 3$ for every minimal prime $\q$ of $J$, 
if $\Spec^\circ\(S/J\)$ has $t-1$ connected components, then $H^{n-1}_J(S)=E_S(K)^{\oplus t-2}$.  
In particular, this holds for $J=J_1\cap\ldots\cap J_{t-1}$.  

Consider the following piece of the Mayer-Vietoris sequence in local cohomology associated to $J$ and $J_t$:
\begin{equation} \label{inductiveMV: e}
\cdots \to H^{n-1}_{J+J_t}(S)\to H^{n-1}_{J}(S)\oplus H^{n-1}_{J_t}(S)\to H^{n-1}_{I}(S)\to H^{n}_{J+J_t}(S)\to 0.
\end{equation}
By Theorem \ref{SVTUnramifiedGeneral: T},  $H^{n-1}_{J_t}(S)=0$, and since $J_i + J_t$ is primary to the maximal ideal $\n$ of $S$ for all $1 \leq i < t$,
 we know that $J+J_t$ also $\n$-primary.
 Then from the inductive hypothesis and \eqref{inductiveMV: e}, we obtain the short exact sequence 
\[
0 \to E_S(K)^{\oplus t-2}\to H^{n-1}_{I}(S)\to E_S(K)\to 0,
\]
which splits since $E_S(K)^{\oplus t-2}$ is injective, so that  $H^{n-1}_I(S) \cong E_S(K)^{\oplus t-1}$.
\end{proof}

\section{Mixed characteristic Lyubeznik numbers of cohomologically isolated singularities} \label{SecIsoSing: S}


We expect the mixed characteristic Lyubeznik numbers to behave, at least in some ways, like the original Lyubeznik numbers.  
The authors have recently shown that for local rings of prime characteristic, they are ``almost always'' equal, in the sense that differences can only occur in small characteristics \cite{InjDim}. 
Moreover, the highest Lyubeznik number and the highest mixed characteristic Lyubeznik number agree in every example where both are computed \cite{NuWi1}.

In this section, we investigate the highest mixed characteristic Lyubeznik numbers of a collection of rings that includes those with isolated singularities.  
In Theorem \ref{highestLyuEqual: T}, we show that this number equals the highest Lyubeznik number if both are defined.  
We also show that it has similar properties to the original Lyubeznik numbers, as they relate to other Lyubeznik numbers (Theorem \ref{highestLowestLyu: T}), and to connectedness properties of spectra (Theorem \ref{connectedComponentsLyu: T}/Theorem \ref{connectedComponentsLyuIntro: T} from the introduction, 
and Theorem \ref{connectednessDimensionLyu: T}).


\subsection{Cohomologically complete intersection ideals on the punctured spectrum of a ring}
We study quotients of regular local rings modulo ideals called \emph{cohomologically complete intersections} on the punctured spectrum of the ring, which were introduced by Hellus and Schenzel \cite{CCI}.  
The vanishing of the local cohomology modules of the ring with support in these ideals mimics that of those with support in complete intersection ideals.

\begin{definition}[Cohomologically complete intersection ideal (on the punctured spectrum)] \label{CCI: D}
Fix an ideal $I$ of a regular 
local ring $(S,\n)$, and let $\ctrOne = \dim(S) -\dim(S/I)$.   
Then $I$ is a \emph{cohomologically complete intersection (CCI) ideal of $S$} if $H^j_I(S)=0$ for all $j> \ctrOne$.
We call $I$ a \emph{CCI ideal on the punctured spectrum of $S$} under the weaker condition that $\Supp_S (H^j_I(S)) \subseteq \{ \n \}$ for all $j > \ctrOne$.
\end{definition}


\begin{example} \label{example1CCI: R}
If $S$ is a regular local ring of characteristic $p>0$ and $I$ is an ideal of $S$ for which $S/I$ is Cohen-Macaulay, then
$I$ is a CCI ideal of $S$ \cite[Theorem III.4.1]{P-S}.  
\end{example}

Given a CCI ideal $I$ on the punctured spectrum of a regular local ring $S$, 
Blickle has called the quotient $S/I$ a \emph{cohomologically isolated singularity ring}  \cite{BlickleCCI}.
A complete local ring that is a complete intersection at every point of its punctured spectrum is isomorphic to such a quotient; so is
a complete local ring with an isolated singularity.  

\begin{example}[Isolated singularities] \label{example2CCI: R}
Suppose that $R$ is a complete local ring of dimension at least three for which $R_\p$ is regular for every non-maximal prime ideal $\p$ of $R$. 
Then $R$ is $S_2$ and $R_1$, it must be normal, and thus a domain.
Use the Cohen structure theorems to write $R \cong S/I$, where $I$ is an ideal of a complete regular local ring $S$ \cite{Cohen}.
Then for every non-maximal prime ideal $\q$ of $S$,
$S_\q/I S_\q$ is a regular ring, so that $I S_\q$ is a  complete intersection ideal of $S_\q$.
Hence, $(H^j_I(S) )_\q \cong H^j_I(S_\q)=0$ for all $j\neq \dim(S)-\dim(R)$, and $I$ is a CCI ideal on the punctured spectrum of $S$.
\end{example}

\begin{proposition} \label{CCIequidimensional: L}
If $I$ is a CCI ideal on the punctured spectrum of a regular local ring $S$, then $S/I$ is equidimensional.
\end{proposition}

\begin{proof}
Let $d=\dim(S/I)$, and list the minimal primes $\p$ of $I$ for which $\dim(S/\p)=d$ as $\p_1,\ldots, \p_\ctrOne$.  
Assume, by way of contradiction, that $S/I$ is not equidimensional, and list the minimal primes $\q$ of $I$ for which $\dim(S/\q)<d$ as $\q_1,\ldots, \q_\ctrTwo$, so that $\ctrTwo \geq 1$.

Let $J_1=\p_1\cap\cdots\cap \p_\ctrOne$ and $J_2=\q_1\cap \cdots \cap \q_\ctrTwo$, so that $\sqrt{I}=J_1\cap J_2.$
Then  $\dim(R/J_1)=d$, and let $e=\dim(R/J_2)$, so that $e<d$. We have the Mayer-Vietoris sequence associated to $J_1$ and $J_2$, 
\[ \cdots \to H^{n-e}_{J_1+J_2}(S)\to H^{n-e}_{J_1}(S)\oplus H^{n-e}_{J_2}(S)\to H^{n-e}_I (S)\to \cdots. \] 
By our choice of $J_1$ and $J_2$, $\dim(S/(J_1+J_2))<e$, so that $H^{n-e}_{J_1+J_2}(S)=0$.
As a consequence,  $H^{n-e}_{J_2}(S)$ injects into $H^{n-e}_I (S)$.  Therefore, 
\[
 \dim \Supp_S(H^{n-e}_I (S))\geq \dim \Supp_S(H^{n-e}_{J_2} (S))=e>0.
\]
Yet, $I$ is a CCI ideal on $\Spec^\circ(S)$, so $\dim \Supp_S(H^{n-e}_I(S)) \leq 0$, a contradiction.
\end{proof}

 The following lemma allows us to reduce the proof of  Theorem \ref{highestLyuEqual: T} to cases of small dimension.


\begin{lemma}\label{AddElement: L}
Given a CCI ideal $I$ on the punctured spectrum of a regular local ring $S$,
if $f\in S$ is not in any minimal prime of  $I$, then $I+fS$ is also a CCI ideal on $\Spec^\circ(S)$.
\end{lemma}


\begin{proof}
Let $n=\dim(S)$ and $d=\dim(S/I)$.  As $f$ belongs to no minimal prime of $I$, we have that $\dim\( S/(I+fS)\) =d-1$.
If $\n$ denotes the maximal ideal of $S$, then $\Supp_S (H^j_I(S) ) \subseteq \{ \n\}$ for $j>n-d$, so $H^j_I(S_f)=0$ for $j >n-d$.  Then 
by the long exact sequence 
\[
\cdots \to H^{j-1}_I(S_f) \to H^j_{I+fS}(S)\to H^j_I(S)\to H^j_I(S_f)\to\cdots,
\]
$H^j_{I+fS}(S) \cong H^j_I(S)$ for $j>n-d+1 = n-(d-1)$, and $\Supp_S ( H^j_{I+fS}(S) ) \subseteq \{ \n \}$ for these values of $j$.
\end{proof}


\subsection{The highest mixed characteristic Lyubeznik number}
Recall that for a complete local ring $R$ containing a field, if $\k$ is a coefficient field of $R$, then there exists a finite separable extension field $L$ of $\k$ for which 
the highest Lyubeznik number of $R$ counts the connected components of the Hochster-Huneke graph of $R \otimes_\k L$ \cite[Theorem 4.2]{H-L}, \cite[Section 2]{LyuInvariants}; see Remark \ref{Rem HHgraph}.
Motivated by this fact, 
with an eye toward comparing the highest traditional and mixed characteristic Lyubeznik numbers, 
we make the following observation about DVR extensions induced by field extensions.

\begin{remark}\label{SameUnderExt: R}
Given an extension of fields  $\kk \subseteq \LL$ of characteristic $p>0$, and 
complete Noetherian DVRs $(V, pV, \kk)$ and $(W, pW, \LL)$ of mixed characteristic $p$, there exists a natural injective local ring homomorphism 
$\phi:V \hookrightarrow W$ \cite[Theorem 29.2]{Matsumura}. 
Then $W$ is a torsion-free  $V$-module.  Since $V$ is also a DVR,
$W$ must be a flat $V$-algebra. 

Let $S =V\llbracket x_1,\ldots,x_n\rrbracket $, $T = W\llbracket x_1,\ldots,x_n\rrbracket$, and 
$\varphi:S\to T$ the injective ring homomorphism satisfying $\varphi|_V=\phi$ and $\varphi(x_i)=x_i$ for each $1 \leq i \leq n$.
Since $\phi$ is flat, so is $\varphi$.
Fix $i, j \geq 0$, and, given an ideal $I$ of $S$, let $\alpha =\dim_\kk\Ext^i_{S}(\kk,H^j_I(S))$, which is finite \cite[Theorem 1]{LyuUMC}, \cite[Theorem 1.1]{NunezPR}.  
Since $\kk$ is finitely presented over $S$ and $T$ is a flat $S$-algebra, 
\[
\LL^{\alpha} \cong  \kk^{\alpha}  \otimes_S T \cong \Ext^i_S(\kk,H^j_I(S))\otimes_{S} T \cong \Ext^i_T(L,H^j_I(S) \otimes_{S} T) \cong \Ext^i_T( \LL,H^j_I(T) ).
\]
Thus, 
$\alpha=\dim_\LL\Ext^i_{T}(\LL,H^j_I(T))$. 
\end{remark}

In the same vein, if $A = \kk\llbracket x_1,\ldots,x_n\rrbracket$ and $B = \LL\llbracket x_1,\ldots,x_n\rrbracket$, then for an ideal $I$ of $A$ and all integers $i, j \geq 0$, 
\[
\dim_\kk\Ext^i_{A}(\kk,H^j_I(A)) = \dim_\LL\Ext^i_{B}(\LL,H^j_I(B)). 
\]


Before returning to our investigation of CCI ideals on punctured spectra, we remove a hypothesis of a result on the agreement of the mixed characteristic and the traditional Lyubeznik numbers in low dimension \cite[Proposition 5.2]{NuWiMixChar};  \emph{loc.\,cit.}\ requires a separably closed residue field. 
This extension serves as the basis case of our inductive proof of Theorem \ref{highestLyuEqual: T}.


\begin{proposition}[cf.\ {\cite[Proposition 5.2]{NuWiMixChar}}] \label{LyuEqualSmallDim: P}
Fix an $n$-dimensional complete unramified regular local ring $S$ of mixed characteristic $p>0$, with residue field $K$.
Assume that $I$ is an ideal of $S$, $p \in I$, and $\dim(S/I)\leq 2$.  If $A = S/pS$, then for all $i, j \geq 0$, 
\[
\lambda_{i,j}(S/I) = \dim_\kk\Ext^i_A (\kk,H^{n-j-1}_{IA} (A)) = \dim_\kk\Ext^i_S(\kk,H^{n-j}_I (S)) = \widetilde{\lambda}_{i,j}(S/I) .
\]
\end{proposition}


\begin{proof} 
Let $R=S/I$, and $d = \dim(R)$.
Under the assumption that $K$ is separably closed, the statement is explicitly justified in the proof of \cite[Proposition 5.2]{NuWiMixChar}, although only stated there for $i = j = d$.
We aim to reduce to this case.  

Now we turn to the general case.
Let $V$ be a complete Noetherian DVR with maximal ideal $pV$ and residue field $K$.
Let $\LL$ denote the separable closure of $K$, and $W$ a corresponding DVR as 
in Remark \ref{SameUnderExt: R}.
By this same remark, for all $i, j \geq 0$, 
\begin{equation} \label{passToSeparable: e}
\lambda_{i,j}(R) =  \lambda_{i,j}(R \widehat{\otimes}_\kk \LL) \ \hbox{ and } \
\widetilde{\lambda}_{i,j}(R) =  \widetilde{\lambda}_{i,j}(R \widehat{\otimes}_\kk \LL).
\end{equation}
Since the residue field of $S \widehat{\otimes}_V W$ is separably closed, the result \cite[Proposition 5.2]{NuWiMixChar} now applies, and  
\begin{equation}
\begin{aligned} \label{separableHighest: e}
\lambda_{i,j}(R \widehat{\otimes}_\kk \LL) &= \dim_\LL \Ext^i_{A\widehat{\otimes}_\kk \LL} (\LL,H^{n-j-1}_I (A \widehat{\otimes}_\kk \LL) ) \\ 
&= \dim_\LL \Ext^i_{S \widehat{\otimes}_V W} (\LL,H^{n-j}_I (S\widehat{\otimes}_V W ) )
&=\widetilde{\lambda}_{i,j}(R \widehat{\otimes}_\kk \LL).
\end{aligned}
\end{equation}
Combining \eqref{passToSeparable: e} and \eqref{separableHighest: e}, we have that $\lambda_{i,j}(R) = \widetilde{\lambda}_{i,j}(R)$.
\end{proof}


To prove two of the main results of this section, Theorems \ref{highestLyuEqual: T} and \ref{connectedComponentsLyu: T} (cf.\ Theorem \ref{connectedComponentsLyuIntro: T} from the introduction), 
we use properties of minimal injective resolutions of certain local cohomology modules in mixed characteristic.  
In particular, we call upon Zhou's work on injective dimension in this setting \cite{Zhou}.

\begin{remark} \label{ZhouProofInjective: R}   
Suppose that $I$ is an ideal of an unramified regular local ring $(S,\n)$ of mixed characteristic.  Zhou proved that for all $j \geq 0$, 
\[
\InjDim_S(H^j_I(S)) \leq \dim \Supp_S(H^j_I(S)) + 1
\]
\cite[Theorem 5.1]{Zhou}; we note that this inequality is further investigated in the authors' recent work \cite{InjDim}.
Suppose that $E^\bullet$ is a minimal injective resolution of $H^j_I(S)$, so that the $i^\text{th}$ cohomology module of the complex
\[
0 \to H^0_\n(E^0) \overset{\varphi_{0}}{\to}  H^0_\n(E^1)  \overset{\varphi_{1}}{\to}  H^0_\n(E^2)  \to \cdots 
\]
computes $H^i_\n(H^j_I(S))$.  As part of Zhou's proof of \emph{loc.\,cit.}, he proves inductively that for all $\ctrOne \geq 0$, the image of $\varphi_\ctrOne$  is an injective $S$-module.    
\end{remark}


Now we can prove our theorem on the agreement of the original and the mixed characteristic Lyubeznik numbers.

\begin{theorem}\label{highestLyuEqual: T}
Let $S$ be a complete unramified regular local ring of mixed characteristic $p>0$.  
Given a CCI ideal $I$ on $\Spec^\circ(S)$ for which $p \in I$,
\[\newLyu_{d,d}(S/I) = \lambda_{d,d}(S/I),\] where $d = \dim(S/I)$.  
Thus, $\newLyu_{d,d}(S/I)$ equals the number of connected components of the Hochster-Huneke graph of the  completion of the strict Henselization of $S/I$.
\end{theorem}


\begin{proof}
Proposition \ref{LyuEqualSmallDim: P} establishes the statement when $\dim(S/I) \leq 2$;
we proceed by induction on $\dim(S/I)$.
Fix $d > 2$, and assume that for every CCI ideal $J$ on $\Spec^\circ(S)$ containing $p$, and for which $\dim(S/J) = d - 1$, we have that $\newLyu_{d-1,d-1}(S/J) = \lambda_{d-1,d-1}(S/J)$.
Let $I$ be a CCI ideal on $\Spec^\circ(S)$ that contains $p$, and such that $\dim(S/I) = d$.

Let $\n$ denote the maximal ideal of $S$.  
Fix $r \in \n$ with the following property:  $r$ is in no minimal prime of $I$, and if $\n$ is not the only associated prime of $H^{n-d+1}_I(S/pS)$ over $S$, 
$r$ is also in no minimal prime of this module. 
Such an element exists by prime avoidance, since $H^{n-d}_I(S/pS)$ has finitely many associated primes \cite[Corollary 2.3]{Huneke}.
The ideal $I+rS$ is a CCI ideal on $\Spec^\circ(S)$ by Lemma \ref{AddElement: L}, and $\dim(S/(I+rS)) =d-1$ 
since $r$ is a nonzerodivisor on $S/I$.
We claim that 
\begin{equation} \label{clm} \newLyu_{d,d}(S/I)=\newLyu_{d-1,d-1}(S/(I+rS)). \end{equation}
In fact, once \eqref{clm} is established, we can conclude that $\lambda_{d,d}(S/I) = \newLyu_{d,d}(S/I)$:   
As $S/I$ is a $d$-dimensional local ring containing a field, due to our choice of $r$,  
we may apply\footnote{ Although Section 2 of \cite{W} sets down some blanket hypotheses on the ring $S/I$ that our ring may not satisfy, here we satisfy all that are needed in the proof of Proposition 2.1.} \cite[Proposition 2.1]{W}, 
which says that $\lambda_{d,d}(S/I)=\lambda_{d-1,d-1}(S/(I+rS))$.
Then by the inductive hypothesis and \eqref{clm}, 
\[
\lambda_{d,d}(S/I) =\lambda_{d-1,d-1}(S/(I+rS))   =\newLyu_{d-1,d-1}(S/(I+rS))  = \newLyu_{d,d}(S/I).
\]

Toward proving \eqref{clm}, consider the long exact sequence 
\begin{equation}  \label{LEScompare: e}
\cdots \to H^i_{I+rS}(S) \to H^i_I(S)\to H^i_I(S_r)\to H^{i+1}_{I+rS}(S)\to \cdots.
\end{equation}
Since the depth of $S$ on $I+ rS$ is $n - \dim(S/(I+rS)) = n - d + 1$,  $H^{n-d}_{I+rS}(S)$ vanishes.
Moreover, $H^{n-d+1}_I(S)_r$ vanishes since $I$ is a CCI ideal on $\Spec^\circ(S)$, and $r \in \n$.  
Thus, from \eqref{LEScompare: e}, we obtain the exact sequence
\begin{equation*} 
0 \to H^{n-d}_I(S)\to H^{n-d}_I(S_r)\to H^{n-d+1}_{I+rS}(S)\to H^{n-d+1}_I(S) \to 0.
\end{equation*}
Let $M = \IM( H^{n-d}_I(S_r)\to H^{n-d+1}_{I+rS}(S) )$, so that we have the short exact sequences
\begin{align*}
&0\to H^{n-d}_I(S)\to H^{n-d}_I(S_r)\to M \to 0  \ \text{ and} \\
&0\to M \to H^{n-d+1}_{I+rS}(S)\to H^{n-d+1}_I(S)\to 0.
\end{align*}
If $K = S/\n$, these, respectively, induce the long exact sequences 
\begin{align}
\label{ExtLES1: e}  
&\cdots\to  \Ext^{i-1}_S(\kk,M)\to \Ext^i_S(\kk,H^{n-d}_I(S))\to 
\Ext^i_S(\kk,H^{n-d}_I(S_r))\to \cdots \\   
\label{ExtLES2: e}  
&\cdots \to \Ext^{i-1}_S(\kk,H^{n-d+1}_I(S)) \to \Ext^i_S(\kk, M) \to
\Ext^i_S(\kk, H^{n-d+1}_{I+rS}(S))\to \cdots.
\end{align}

We have that $\Ext^i_S(\kk,H^{n-d}_I(S_r))=0$ for all $i\geq 0$.
Additionally, since $\Supp_S H^{n-d+1}_I(S) \subseteq \{\n\}$, we have that 
\[ \InjDim_S H^{n-d+1}_I(S)\leq \dim \Supp_S H^{n-d+1}_I(S)+1 \leq 1 \]
\cite[Theorem 5.1]{Zhou}, so that $\Ext^i_S(\kk,H^{n-d+1}_I(S))=0$ for all $i\geq 2$.
Noting that $d>2$, these observations, along with \eqref{ExtLES1: e} and \eqref{ExtLES2: e}, imply that
\[
\Ext^{d}_S(\kk,H^{n-d}_I(S))\cong \Ext^{d-1}_S(\kk,M)
\cong \Ext^{d-1}_S(\kk,H^{n-d+1}_{I+rS}(S)).
\]
Taking $\kk$-vector space dimensions of the left- and right-hand sides, we conclude that \eqref{clm} holds, completing the proof that $\lambda_{d,d}(S/I) = \newLyu_{d,d}(S/I)$.

 The final statement now follows because $\lambda_{d,d}(S/I)$ equals the number of components of the Hochster-Huneke graph of the completion of the strict Henselization of $S/I$ \cite[Theorem 1.3]{LyuInvariants},\ \cite[Main Theorem]{W}. 
\end{proof}


We point out that in light of Theorem \ref{highestLyuEqual: T}, it would be natural to investigate whether the highest Lyubeznik number agrees with the highest mixed characteristic Lyubeznik number more generally \cite[Question 6.13]{NuWiMixChar}.

The remainder of the paper is dedicated to establishing several properties of the highest mixed characteristic Lyubeznik number, motivated by analogous properties of their equal characteristic counterpart.  
In order to do so, we first study the structure of certain local cohomology modules in mixed characteristic. 


\begin{remark} \label{surjectiveInjecitve: R}
 Suppose that $(S, \n)$ is an  unramified regular local ring of mixed characteristic $p>0$, and $M$ a local cohomology module of $S$ supported at $\n$.
If multiplication by $p$ is surjective on $M$, then $M$ must be an injective $S$-module \cite[Lemma 4.2]{NuWiMixChar}.

One way we apply this fact is the following:  Assume that $I$ is a CCI on the punctured spectrum of such a ring $S$, and $\dim(S/I) \geq 2$.
Then $H^{n-1}_I(S)$ is either zero or supported only at $\n$.  
If $A = S/pS$, then  $\dim(A/I A) \geq 1$, so that the short exact sequence 
\begin{equation} \label{SESTimesP: e}
0\to S\FDer{p} S\to A\to 0
\end{equation}
induces the long exact sequence 
\begin{equation} \label{usefulLES: e}
\cdots \to H^{n-2}_I(S)\FDer{p} H^{n-2}_I(S)\to H^{n-2}_I(A)\to H^{n-1}_I(S)\FDer{p} H^{n-1}_I(S)\to 0, 
\end{equation}
where the final zero is a consequence of the HLVT \cite[Theorem 3.1]{HartshorneCD}.  Thus, $H^{n-1}_I(S)$ is an injective $S$-module. 
\end{remark}


To prove the following structural theorem, we briefly call upon the theory of (left) modules over rings of differential operators.
We do not review details of this theory, and refer the reader to the reference \cite{Bj1}.
Given a subring $T$ of a ring $R$, $D(R,T)$ denotes the ring of $T$-linear differential operators on $R$.


\begin{lemma}\label{LCnMinus2InjectiveP: L}
Let $S$ be an $n$-dimensional complete unramified regular local ring of mixed characteristic $p>0$, with a separably closed residue field.
Let $I$ be a CCI ideal on $\Spec^\circ(S)$ that contains $p$, for which $\dim(S/I) \geq 3$.
Then $H^{n-2}_I(S)$ is an injective $S$-module. 
\end{lemma}

\begin{proof}
 Let $\n$ denote the maximal ideal of $S$, let $K = S/\n$, and let $A = S/pS$.
If $\ctrOne$ is the number of connected components of the punctured spectrum of $S/I \cong A/IA$, 
it is well known to experts that $H^{n-2}_I(A) \cong E_A(K)^{\oplus \ctrOne - 1}$ as $A$-modules.
Since $S/I$ is equidimensional by Lemma \ref{CCIequidimensional: L}, we can also conclude that $H^{n-1}_I(S) \cong E_S(K)^{\oplus \ctrOne - 1}$ as $S$-modules by Proposition \ref{connectedComponents: P}.
Therefore, the kernel $N$ of multiplication by $p$ on $H^{n-1}_I(S)$
in \eqref{usefulLES: e} is isomorphic with $E_A(K)^{\oplus \ctrOne - 1}$.

Consider the induced surjection 
\begin{equation} \label{inducedMap: e}
H^{n-2}_I(A) \onto N \subseteq H^{n-1}_I(S).
\end{equation}
Identify $S\cong V \llbracket x_1,\ldots,x_{n-1} \rrbracket $, where $(V,pV,K)$ is a complete Noetherian DVR of mixed characteristic $p$ \cite{Cohen}.
If $D(S,V)$ is the ring of $V$-linear differential operators on $S$,  then \eqref{usefulLES: e} is a sequence of (left) $D(S,V)$-modules.
Moreover, the induced surjection \eqref{inducedMap: e} is a map of $D(A,K)$-modules \cite[(2.0.1)]{Todos}.  
Since $E_A(K)$ is a simple $D(A,K)$-module and $H^{n-2}_I(S) \cong N \cong E_A(K)^{\oplus \ctrOne -1}$, both $H^{n-2}_I(S)$ and $N$ have length $\ctrOne-1$ as $D(S,K)$-modules.
Since then the kernel of \eqref{inducedMap: e} must have length zero, this map must be an isomorphism.  

Looking back at \eqref{usefulLES: e}, we conclude that multiplication by $p$ on $H^{n-2}_I(S)$ is surjective.
Since $I$ is a CCI ideal on $\Spec^\circ(S)$, $H^{n-2}_I(S)$ either vanishes or is supported only at $\n$, and so it must be injective \cite[Lemma 4.2]{NuWiMixChar}. 
\end{proof}


 The following lemma's aim is to remove the hypothesis in Lemma \ref{LCnMinus2InjectiveP: L} that the specified ideal must contain $p$.

\begin{proposition}\label{LCnMinus2Injective: P}
Let $S$ be an $n$-dimensional complete unramified regular local ring of mixed characteristic, with a separably closed residue field.
Let $I$ be a CCI ideal on $\Spec^\circ(S)$ for which 
$\dim(S/I)\geq 4$.
Then  $H^{n-2}_I(S)$ is an injective $S$-module.
\end{proposition}


\begin{proof}
Let $\n$ denote the maximal ideal of $S$, $p = \Char(S/\n)$, and $d = \dim(S/I)$.  
Since $p \in \n$ and  $\Supp_S( H^j_I(S)) \subseteq \{ \n \}$ for $j > n-d$, $H^j_I(S_p) \cong ( H^j_I(S) )_p = 0$.
Then from the long exact sequence \[ \cdots \to H^{j-1}_I(S_p)\to H^j_{I+pS}(S) \to H^j_I(S) \to H^j_I(S_p) \to \cdots,\] we see that
\begin{equation} \label{isomLCnoP: e}
H^j_{I+p S}(S)\cong H^j_I(S) \ \ \text{for all} \ \ j > n-d+1 = n - (d-1).
\end{equation}

The dimension of $S/(I+pS)$ is either $d-1$ or $d$.   First assume the former, that $\dim\( S/(I+pS)\) =d-1 \geq 3$.  
In this case, since $I$ is a CCI ideal on $\Spec^\circ(S)$,  \eqref{isomLCnoP: e} implies that  $I + pS$ is one as well.
Then $H^{n-2}_I(S) \cong H^{n-2}_{I+pS}(S)$ is injective by Lemma \ref{LCnMinus2InjectiveP: L}.

Alternatively, suppose that $\dim\( S/(I+pS)\) =d \geq 4$.
Let $J$ denote the intersection of the (finitely many) minimal primes of $H^{n-d+1}_{I+pS}(S)$ \cite[Theorem 1]{LyuUMC}.
 Then 
 \[
 \Supp_S ( H^{n-d+1}_{I+pS}(S) ) = \V(J) \subseteq \V(I+pS),
 \]
and
\[
\dim ( S/J ) = \dim\Supp_S ( H^{n-d+1}_{I+pS}(S) ) \leq n - (n-d+1)=  d-1.\]
Then by prime avoidance, there exists $f \in J$ not in any minimal prime $\q$ of $I+pS$ for which $\dim(S/\q)=d$.
By our choice of $J$,  $H^{n-d+1}_{I+pS}(S_f)\cong H^{n-d+1}_{I+pS}(S)_f=0$. 
Then noting \eqref{isomLCnoP: e} as well, we know that $H^j_{I+pS}(S_f)=0$ for all $j > n-d$, which means that 
\begin{equation} \label{isomLCnoPf: e}
H^j_{I+(f,p) S}(S) \cong H^j_{I+p S}(S) \ \ \text{for all} \ \  j > n-d+1
\end{equation}
by the long exact sequence 
\[
\cdots \to H^{j-1}_{I+pS}(S_f) \to H^j_{I+(f,p) S}(S)\to H^j_{I+p S}(S) \to H^j_{I+pS}(S_f) \to \cdots.
\]

By \eqref{isomLCnoP: e} and \eqref{isomLCnoPf: e}, $H^j_{I + (f,p)S}(S) \cong H^j_I(S)$ for all $j > n-d+1$, and since $\dim( S/(I+(f,p)S)) =d-1 \geq 3$ 
by our choice of $f$, $I+(f,p)S$ is a CCI ideal on $\Spec^\circ(S)$ since $I$ is one.
Then $H^{n-2}_I(S) \cong H^j_{I + (f,p)S}(S)$ is an injective $S$-module by Lemma \ref{LCnMinus2InjectiveP: L}.
\end{proof}


If $S$ is an $n$-dimensional regular local ring containing a field, and $I$ is an ideal of $S$ for which $\dim(S/I) =d$, then the injective dimension of $H^{n-d}_I(S)$ is $d$ \cite[(4.4iii)]{LyuDMod}.
The same is true if $S$ is an $n$-dimensional complete unramified regular local ring of mixed characteristic \cite[Theorem 4.9]{NuWiMixChar}, and we call upon this fact in the proofs of the next two lemmas.  
The first relies on a spectral sequence argument, and is useful in proving Theorem \ref{highestLowestLyu: T}, a duality theorem on mixed characteristic Lyubeznik numbers.


\begin{lemma}\label{iteratedLC: L}
Let $(S, \n, K)$ be an $n$-dimensional complete unramified regular local ring of mixed characteristic, and 
let $I$ be a CCI ideal on $\Spec^\circ(S)$ for which 
 $\dim(S/I)=d$ and $d \geq 2$.  
Then 
\begin{enumerate}
\item \label{iteratedLC 1: i} $H^i_\n ( H^{n-d}_I(S) ) \cong H^{n-d+i-1}_I(S)$ for all $i < d$,  
\item \label{iteratedLC 2: i} $H^d_\n ( H^{n-d}_I(S) ) \cong H^{n-1}_I(S)\oplus E_S(\kk)$, 
\end{enumerate}
and all other $H^i_\n ( H^d_I(S)) = 0$.
\end{lemma}


\begin{proof}
Consider the spectral sequence $E_2^{i,j} = H^i_\n ( H^j_I(S) ) \underset{i}{\implies} H^{i+j}_\n (S)$ \cite[Proposition 1.4]{HartshorneLC}.  
By our assumptions on $I$, the HLVT \cite[Theorem 3.1]{HartshorneCD}, and the fact that $\InjDim_S H^{n-d}_I(S) = d$ \cite[Theorem 4.9]{NuWiMixChar},  all $E_2^{i,j} = 0$ except possibly: 
\begin{itemize}
\item $E_2^{0, j} = H^0_\n(H^j_I(S)) = H^j_I(S)$, where  $n - d + 1 \leq j \leq n-1$, and 
\item $E_2^{i, n-d} = H^i_\n(H^{n-d}_I(S))$, where $0 \leq i \leq d$. 
\end{itemize}
The $E_2$-sheet is concentrated in row $n-d$ and column zero.  
This means that the only possibly nonzero differentials are, for $r \geq 2$, \[\partial_r^{0,n-d+r-1}: E_2^{0, n-d+r-1} \to E_2^{r, n-d}. \] 
Since $H^\ctrOne_\m(S) \neq 0$ if and only if $\ctrOne = n$, each $\partial_r^{0,n-d+r-1}$ must be an isomorphism for $r \neq d$, and
$  H^{n-d+r-1}_I(S) = H^0_\n (H^{n-d+r-1}_I(S)) \cong  H^r_\n (H^{n-d}_I(S))$, establishing \eqref{iteratedLC 1: i}.
We also know that $\partial_d^{0,n-1}: E_d^{0,n-1} = H^{n-1}_I(S) \to E_d^{d,n-d} = H^d_\n (H^{n-d}_I(S))$ must be injective.
Moreover, since $H^n_\n(S) \cong E_S(\kk)$, if $\partial = \partial_d^{0, n-1}$, there is a short exact sequence
\begin{equation*} 
0 \to H^{n-1}_I(S) \overset{\partial}{\to} H^d_\n (H^{n-d}_I(S) ) \to E_S(\kk) \to 0. \label{shortexactspectral: e}.
\end{equation*}
By Remark \ref{surjectiveInjecitve: R}, $H^{n-1}_I(S)$ is injective, so this sequence splits, and \eqref{iteratedLC 2: i} holds.
\end{proof}


The following result helps us understand the structure of a minimal injective resolution of certain local cohomology modules in mixed characteristic.  

\begin{lemma} \label{InjResn: L}
Take an ideal $I$ of an $n$-dimensional complete unramified regular local ring  $S$ of mixed characteristic, and let $d = \dim(S/I)$. Since $\InjDim H^{n-d}_I(S) = d$  \cite[Theorem 4.9]{NuWiMixChar}, a minimal injective resolution of $H^{n-d}_I(S)$ has the form \[ E^\bullet = (E^0 \to E^1 \to \cdots \to E^d \to 0 ). \]
Then $E^d$ is supported only at the maximal ideal of $S$.
\end{lemma}

\begin{proof}
For a prime ideal $\q$ such that $I \subseteq \q \subsetneq \m$,
\[ \InjDim H^{n-d}_I(S_\q) \leq \dim(S_\q / I S_\q) \leq d - 1. \]
Therefore, since 
\[
E^\bullet_\q = \((E^0)_\q \to (E^1)_\q \to \ldots \to (E^d)_\q \to 0 \)
\]
 is a minimal injective resolution of $H^{n-d}_I(S_\q)$,
we have that $\(E^d\)_\q=0$. 
\end{proof}


We take advantage of the following duality theorem to relate the highest mixed characteristic Lyubeznik number to the number of connected components of the punctured spectrum.
See \cite[Theorem]{GarciaSabbah}. and \cite[Theorem 1.2]{B-B} for similar statements about Lyubeznik numbers in equal characteristic.


\begin{theorem} \label{highestLowestLyu: T}
Let $I$ be a CCI ideal on the punctured spectrum of a complete unramified regular local ring $S$ of mixed characteristic, with a separably closed residue field.  If $d:=\dim(S/I)\geq 4$, then \[\newLyu_{d,d}(S/I)=\newLyu_{0,1}(S/I) + 1.\]
\end{theorem} 


\begin{proof}
 Let $n = \dim(S)$, let $\n$ denote the maximal ideal of $S$, and let $K = S/\n$.
If $E^\bullet$ is a minimal injective resolution of $H^{n-d}_I(S)$, then $E^i = 0$ for $i > d$  \cite[(4.4iii)]{LyuDMod}.
Consider the complex
\begin{equation} \label{computeHH: e}
\cdots \to H^0_\n(E^{d-2}) \overset{\varphi_{d-2}}{\to} H^0_\n(E^{d-1}) \overset{\varphi_{d-1}}{\to} H^0_\n(E^d) \to  0,
\end{equation}
whose $i^\text{th}$ cohomology module is $H^i_\n(H^{n-d}_I(S))$.
Let $M$ denote $\IM(\psi)$.

Noting Remark \ref{ZhouProofInjective: R},  $M$ is an injective $S$-module  \cite[Theorem 5.1]{Zhou}).
Therefore, $H^0_\n(E^{d-1}) \cong M \oplus N$, where $N \cong E_S(\kk)^{\oplus \ctrOne}$ for some $\ctrOne \geq 0$ since $H^0_\n(E^{d-1})$ is an injective $S$-module supported at $\n$.  
Making this identification in \eqref{computeHH: e},  \[H^{d-1}_\n H^{n-d}_I(S) \cong \Ker(\varphi_{d-1})/M \cong \Ker(\varphi_{d-1}|_N) \subseteq N \cong E_S(K)^{\oplus \ctrOne},\]
so that for some $S$-module $P$,
\[N \cong E_S(K)^{\oplus \ctrOne} \cong H^{d-1}_\n H^{n-d+1}_I(S) \oplus P,\]
and under this further identification in \eqref{computeHH: e}, $\Ker(\varphi_{d-1}|_P) = 0$.

By Lemma \ref{iteratedLC: L}, $H^{d-1}_\n (H^{n-d}_I(S))\cong H^{n-2}_I(S)$, and is then injective by Proposition \ref{LCnMinus2Injective: P}; 
this means that $P \cong  E_S(K)^{\oplus \ctrTwo}$ for some $0 \leq \ctrTwo \leq \ctrOne$.
Then since $E^\bullet$ is a minimal resolution, $P \cap M = 0$, and $\Ker(\varphi_{d-1}|_P) = 0$, we must have that $\ctrTwo=0$, and $P=0$.
This implies that $\varphi_{d-1} = 0$, so that $H^d_\n(H^{n-d}_I(S)) \cong H^0_\n(E^d)$, which equals $E^d$ by Lemma \ref{InjResn: L}.
Then by Lemma \ref{iteratedLC: L}, $E^d  \cong H^{n-1}_I(S) \oplus E_S(K)$.  
As $H^{n-1}_I(S)$ is injective by Remark \ref{surjectiveInjecitve: R}, and supported only at $\n$ by our assumptions on $I$, 
$\newLyu_{0,1}(S/I) = \dim_K \Hom_S(K, H^{n-1}_I(S))$. 
Now, we can finally conclude that 
\[
E_S(\kk)^{\oplus \newLyu_{d,d}(S/I) } \cong E^d \cong E_S(\kk)^{\oplus \newLyu_{0,1}(S/I)+1},
\]
so that  $\newLyu_{d,d}(S/I) = \newLyu_{0,1}(S/I)+1$.
\end{proof}


The following result relates the highest mixed characteristic Lyubeznik number to the number of components of the punctured spectrum. 
 Note that Theorem \ref{connectedComponentsLyuIntro: T} from the introduction is a straightforward consequence, after writing the ring $R$ in its hypothesis as the quotient 
of a complete unramified regular local ring $S$ of mixed characteristic by a cohomologically complete intersection ideal $I$ on its punctured spectrum. 

An analogous statement for the Lyubeznik numbers is known for rings containing a field, applying Poincar\'e duality in characteristic zero \cite{GarciaSabbah}, and spectral sequences in positive characteristic \cite{B-B,BlickleCCI}.

\begin{theorem}[cf.\ Theorem \ref{connectedComponentsLyuIntro: T}] \label{connectedComponentsLyu: T}
Let $S$ be a complete unramified regular local ring of mixed characteristic, with a separably closed residue field.  
Suppose that $I$ is a CCI ideal on $\Spec^\circ(S)$ for which $d:= \dim(S/I) \geq 4$.
Then $\Spec^\circ(S/I)$ has $\newLyu_{d,d}(S/I)$ connected components.
\end{theorem}

\begin{proof}
Since $S/I$ is equidimensional by Lemma \ref{CCIequidimensional: L}, 
the statement follows immediately from Proposition \ref{connectedComponents: P} and Theorem \ref{highestLowestLyu: T}.
\end{proof}



Our final result again relates the highest mixed characteristic Lyubeznik number to connectedness properties of the spectrum of the ring.
It says that if this number is as small as possible, then the spectrum of the ring must have high connectedness dimension.
The analogous statement in equal characteristic holds:  
If the highest Lyubeznik number of a complete, equidimensional local ring containing a field equals one, then the  
 Hochster-Huneke graph of the completion of the strict Henselization of the ring is connected, 
 so that the Hochster-Huneke graph of the ring itself is also connected (see Remark \ref{Rem HHgraph}).
 Thus, the ring has connectedness dimension at least one less than its dimension \cite[Theorem 3.6]{HHgraph}.


\begin{theorem}\label{connectednessDimensionLyu: T}
Let $S$ be a complete unramified regular local ring of mixed characteristic, and let $I$ be an ideal of $S$ for which $R=S/I$ is equidimensional and has dimension $d\geq 2$.
If $\newLyu_{d,d}(R)=1$, then $c(R) \geq d - 1$; i.e., for every closed subset $Z$ of $\Spec(R)$ of dimension at most $d-2$, $\Spec(R) \setminus Z$ is connected.
\end{theorem}


\begin{proof}
By way of contradiction, suppose that there exists a closed subset $Z$ of $\Spec(R)$ for which $\Spec(R) \setminus  Z$ is disconnected and $\dim(Z) \leq d-2$.
Fix an ideal $\fa$ of $S$ for which $Z=\V(\fa R)$ in $\Spec(R)$.
Then there exist ideals $J_1$ and $J_2$, such that
$\sqrt{J_1\cap J_2}=\sqrt{I}$ and $\sqrt{J_1+J_2}  = \sqrt{\ideala}$.
Since $S/I$ is equidimensional, $\dim(S/J_1) = \dim(S/J_2) = d$.

Consider the Mayer-Vietoris sequence in local cohomology associated to $J_1$ and $J_2$:
\[
0\to H^{n-d}_{J_1}(S)\oplus H^{n-d}_{J_2}(S)\to H^{n-d}_{I}(S)\to H^{n-d+1}_{J_1+J_2}(S)\to \cdots.
\]
Since $\dim\(S/(J_1+J_2)\) \leq d - 2$, 
$H^{n-d+1}_{J_1+J_2}(S)=0$, so that 
$
H^{n-d}_{I}(S)\cong H^{n-d}_{J_1}(S)\oplus H^{n-d}_{J_2}(S)
$, and if $K$ is the residue field of $S$, then
\[\Ext^{d}_S(K,H^{n-d}_{I}(S))\cong \Ext^{d}_S(K,H^{n-d}_{J_1}(S))\oplus \Ext^{d}_S(K,H^{n-d}_{J_2}(S)).
\]
Taking $K$-vector space dimensions, we have that $\newLyu_{d,d}(S/I)=\newLyu_{d,d}(S/J_1)+\newLyu_{d,d}(S/J_2)$.  
However, the highest mixed characteristic Lyubeznik number is positive \cite[Proposition 3.12]{NuWiMixChar}, 
$\newLyu_{d,d}(S/J_1)+\newLyu_{d,d}(S/J_2)\geq 2$, a contradiction.  
\end{proof}


\section*{Acknowledgments}
We would like to thank Alessandro De Stefani for conversations about rings with isolated singularities.  
We are also grateful to the National Science Foundation (NSF): 
the first author has been partially supported by NSF Postdoctoral Research Fellowship DMS-1304250 and NSF DMS-1600702, 
the second by NSF DMS-1502282,
the third by NSF DMS-1068190,
and the last by NSF DMS-1623035 (DMS-1501404 before transfer).


\bibliographystyle{alpha}
\bibliography{References}

\newcommand{\etalchar}[1]{$^{#1}$}
\begin{thebibliography}{HNBPW15}

\bibitem[BB05]{B-B}
Manuel Blickle and Raphael Bondu.
\newblock Local cohomology multiplicities in terms of \'etale cohomology.
\newblock {\em Ann. Inst. Fourier (Grenoble)}, 55(7):2239--2256, 2005.

\bibitem[BBL{\etalchar{+}}14]{Todos}
Bhargav Bhatt, Manuel Blickle, Gennady Lyubeznik, Anurag~K. Singh, and Wenliang
  Zhang.
\newblock Local cohomology modules of a smooth {$\Bbb{Z}$}-algebra have
  finitely many associated primes.
\newblock {\em Invent. Math.}, 197(3):509--519, 2014.

\bibitem[Bj{\"o}79]{Bj1}
Jan-Erik Bj{\"o}rk.
\newblock {\em Rings of differential operators}, volume~21 of {\em
  North-Holland Mathematical Library}.
\newblock North-Holland Publishing Co., Amsterdam, 1979.

\bibitem[Bli07]{BlickleCCI}
Manuel Blickle.
\newblock Lyubeznik's invariants for cohomologically isolated singularities.
\newblock {\em J. Algebra}, 308(1):118--123, 2007.

\bibitem[BS98]{BroSharp}
Markus~P. Brodmann and Rodney~Y. Sharp.
\newblock {\em Local cohomology: an algebraic introduction with geometric
  applications}, volume~60 of {\em Cambridge Studies in Advanced Mathematics}.
\newblock Cambridge University Press, Cambridge, 1998.

\bibitem[Coh46]{Cohen}
I.~S. Cohen.
\newblock On the structure and ideal theory of complete local rings.
\newblock {\em Trans. Amer. Math. Soc.}, 59:54--106, 1946.

\bibitem[DT15]{DaoTakagi}
Hailong Dao and Shunsuke Takagi.
\newblock On the relationship between depth and cohomological dimension.
\newblock {\em Preprint}, 2015.

\bibitem[Fal80a]{FalNagoya}
Gerd Faltings.
\newblock A contribution to the theory of formal meromorphic functions.
\newblock {\em Nagoya Math. J.}, 77:99--106, 1980.

\bibitem[Fal80b]{FalPRIMS}
Gerd Faltings.
\newblock Some theorems about formal functions.
\newblock {\em Publ. Res. Inst. Math. Sci.}, 16(3):721--737, 1980.

\bibitem[Fal80c]{Fa1}
Gerd Faltings.
\newblock \"{U}ber locale kohomologiegruppen hoher ordnung.
\newblock {\em J. Reine Angewandte Math}, pages 43--51, 1980.

\bibitem[GLS98]{GarciaSabbah}
R.~Garc{\'{\i}}a~L{{\'o}}pez and C.~Sabbah.
\newblock Topological computation of local cohomology multiplicities.
\newblock {\em Collect. Math.}, 49(2-3):317--324, 1998.
\newblock Dedicated to the memory of Fernando Serrano.

\bibitem[Gro68]{SGA2}
Alexander Grothendieck.
\newblock {\em Cohomologie locale des faisceaux coh\'erents et th\'eor\`emes de
  {L}efschetz locaux et globaux {$(SGA$} {$2)$}}.
\newblock North-Holland Publishing Co., Amsterdam; Masson \& Cie, \'Editeur,
  Paris, 1968.
\newblock Augment{\'e} d'un expos{\'e} par Mich{\`e}le Raynaud, S{\'e}minaire
  de G{\'e}om{\'e}trie Alg{\'e}brique du Bois-Marie, 1962, Advanced Studies in
  Pure Mathematics, Vol. 2.

\bibitem[Har67]{HartshorneLC}
Robin Hartshorne.
\newblock {\em Local cohomology}, volume 1961 of {\em A seminar given by A.
  Grothendieck, Harvard University, Fall}.
\newblock Springer-Verlag, Berlin, 1967.

\bibitem[Har68]{HartshorneCD}
Robin Hartshorne.
\newblock Cohomological dimension of algebraic varieties.
\newblock {\em Ann. of Math. (2)}, 88:403--450, 1968.

\bibitem[HH94]{HHgraph}
Melvin Hochster and Craig Huneke.
\newblock Indecomposable canonical modules and connectedness.
\newblock In {\em Commutative algebra: syzygies, multiplicities, and birational
  algebra (South Hadley, MA, 1992)}, volume 159 of {\em Contemp. Math.}, pages
  197--208. Amer. Math. Soc., Providence, RI, 1994.

\bibitem[HL90]{H-L}
Craig Huneke and Gennady Lyubeznik.
\newblock On the vanishing of local cohomology modules.
\newblock {\em Invent. Math.}, 102(1):73--93, 1990.

\bibitem[HNBPW15]{InjDim}
Daniel~J. Hern\'andez, Luis N\'u{\~n}ez-Betancourt, Juan~F. P\'erez, and
  Emily~E. Witt.
\newblock Lyubeznik numbers and the injective dimension of local cohomology
  modules in mixed characteristic.
\newblock {\em Preprint}, 2015.

\bibitem[HS93]{Huneke}
Craig~L. Huneke and Rodney~Y. Sharp.
\newblock Bass numbers of local cohomology modules.
\newblock {\em Trans. Amer. Math. Soc.}, 339(2):765--779, 1993.

\bibitem[HS06]{SwHu}
Craig Huneke and Irena Swanson.
\newblock {\em Integral closure of ideals, rings, and modules}, volume 336 of
  {\em London Mathematical Society Lecture Note Series}.
\newblock Cambridge University Press, Cambridge, 2006.

\bibitem[HS08]{CCI}
Michael Hellus and Peter Schenzel.
\newblock On cohomologically complete intersections.
\newblock {\em J. Algebra}, 320(10):3733--3748, 2008.

\bibitem[ILL{\etalchar{+}}07]{TwentyFourHours}
Srikanth~B. Iyengar, Graham~J. Leuschke, Anton Leykin, Claudia Miller, Ezra
  Miller, Anurag~K. Singh, and Uli Walther.
\newblock {\em Twenty-four hours of local cohomology}, volume~87 of {\em
  Graduate Studies in Mathematics}.
\newblock American Mathematical Society, Providence, RI, 2007.

\bibitem[Kaw00]{Kawasaki2}
Ken-ichiroh Kawasaki.
\newblock On the {L}yubeznik number of local cohomology modules.
\newblock {\em Bull. Nara Univ. Ed. Natur. Sci.}, 49(2):5--7, 2000.

\bibitem[KLZ14]{ExtHar}
Mordechai Katzman, Gennady Lyubeznik, and Wenliang Zhang.
\newblock An extension of a theorem of {H}artshorne.
\newblock {\em Preprint}, arXiv:{1408.0858} [math.AC], 2014.

\bibitem[Lyu85]{LySomeAlgebraicSets}
Gennady Lyubeznik.
\newblock Some algebraic sets of high local cohomological dimension in
  projective space.
\newblock {\em Proc. Am. Math. Soc.}, 95:9--10, 1985.

\bibitem[Lyu93]{LyuDMod}
Gennady Lyubeznik.
\newblock Finiteness properties of local cohomology modules (an application of
  {$D$}-modules to commutative algebra).
\newblock {\em Invent. Math.}, 113(1):41--55, 1993.

\bibitem[Lyu00]{LyuUMC}
Gennady Lyubeznik.
\newblock Finiteness properties of local cohomology modules for regular local
  rings of mixed characteristic: the unramified case.
\newblock {\em Comm. Algebra}, 28(12):5867--5882, 2000.
\newblock Special issue in honor of Robin Hartshorne.

\bibitem[Lyu06a]{LyuInvariants}
Gennady Lyubeznik.
\newblock On some local cohomology invariants of local rings.
\newblock {\em Math. Z.}, 254(3):627--640, 2006.

\bibitem[Lyu06b]{LyuVan}
Gennady Lyubeznik.
\newblock On the vanishing of local cohomology in characteristic {$p>0$}.
\newblock {\em Compos. Math.}, 142(1):207--221, 2006.

\bibitem[Lyu07]{LyuLC}
Gennady Lyubeznik.
\newblock On some local cohomology modules.
\newblock {\em Adv. Math.}, 213(2):621--643, 2007.

\bibitem[Mat80]{Matsumura}
Hideyuki Matsumura.
\newblock {\em Commutative algebra}, volume~56 of {\em Mathematics Lecture Note
  Series}.
\newblock Benjamin/Cummings Publishing Co., Inc., Reading, Mass., second
  edition, 1980.

\bibitem[NB13]{NunezPR}
Luis N{\'u}{\~n}ez-Betancourt.
\newblock Local cohomology modules of polynomial or power series rings over
  rings of small dimension.
\newblock {\em Illinois J. Math.}, 57(1):279--294, 2013.

\bibitem[NBW13]{NuWiMixChar}
Luis N{\'u}{\~n}ez-Betancourt and Emily~E. Witt.
\newblock Lyubeznik numbers in mixed characteristic.
\newblock {\em Math. Res. Lett.}, 20(6):1125--1143, 2013.

\bibitem[NBW14]{NuWi1}
Luis N{\'u}{\~n}ez-Betancourt and Emily~E. Witt.
\newblock Generalized {L}yubeznik numbers.
\newblock {\em Nagoya Math. J.}, 215:168--202, 2014.

\bibitem[Ogu73]{Ogus}
Arthur Ogus.
\newblock Local cohomological dimension of algebraic varieties.
\newblock {\em Ann. of Math. (2)}, 98:327--365, 1973.

\bibitem[PS73]{P-S}
C.~Peskine and L.~Szpiro.
\newblock Dimension projective finie et cohomologie locale. {A}pplications \`a
  la d\'emonstration de conjectures de {M}. {A}uslander, {H}. {B}ass et {A}.
  {G}rothendieck.
\newblock {\em Inst. Hautes \'Etudes Sci. Publ. Math.}, (42):47--119, 1973.

\bibitem[Swi15]{SwitalaNonsingular}
Nicholas Switala.
\newblock Lyubeznik numbers for nonsingular projective varieties.
\newblock {\em Bull. Lond. Math. Soc.}, 47(1):1--6, 2015.

\bibitem[Var13]{VarbaroCD}
Matteo Varbaro.
\newblock Cohomological and projective dimensions.
\newblock {\em Compos. Math.}, 149(7):1203--1210, 2013.

\bibitem[Wal01]{Walther2}
Uli Walther.
\newblock On the {L}yubeznik numbers of a local ring.
\newblock {\em Proc. Amer. Math. Soc.}, 129(6):1631--1634 (electronic), 2001.

\bibitem[Zha07]{W}
Wenliang Zhang.
\newblock On the highest {L}yubeznik number of a local ring.
\newblock {\em Compos. Math.}, 143(1):82--88, 2007.

\bibitem[Zho98]{Zhou}
Caijun Zhou.
\newblock Higher derivations and local cohomology modules.
\newblock {\em J. Algebra}, 201(2):363--372, 1998.

\end{thebibliography}


\vspace{.5cm}

{\footnotesize
\noindent \textsc{Department of Mathematics, University of Kansas, Lawrence, KS 66045} \\ \indent \emph{Email address}: {\tt hernandez@ku.edu}

\vspace{.25cm}

\noindent \textsc{Centro de Investigaci\'on en Matem\'aticas, Guanajuato, Gto., M\'exico} \\ \indent \emph{Email address}:  {\tt luisnub@cimat.mx} 

\vspace{.25cm}

\noindent \textsc{Department of Mathematics and Statistics, Georgia State University, Atlanta, GA} {30303} \\ \indent  \emph{Email address}:  {\tt jperezvallejo@gsu.edu} 

\vspace{.25cm}

\noindent \textsc{Department of Mathematics, University of Kansas, Lawrence, KS 66045} \\ \indent \emph{Email address}:  {\tt witt@ku.edu} 
}

\end{document}